\documentclass[leqno]{amsart}
\usepackage{amsmath}
\usepackage{amssymb}
\usepackage{amsthm}
\usepackage{graphicx}
\usepackage{enumerate}
\usepackage[mathscr]{eucal}
\theoremstyle{plain}
\newtheorem{theorem}{Theorem}[section]
\newtheorem{prop}[theorem]{Proposition}
\newtheorem{cor}[theorem]{Corollary}

\theoremstyle{definition}
\newtheorem{definition}[theorem]{Definition}
\newtheorem{remark}[theorem]{Remark}

\newtheorem{example}[theorem]{Example}
\newtheorem*{acknowledgement}{\textnormal{\textbf{Acknowledgements}}}

\usepackage[pagewise]{lineno}

\begin{document}

\title[Extreme contractions on finite-dimensional polygonal spaces]{Extreme contractions on finite-dimensional polygonal Banach spaces}
\author[Debmalya Sain, Anubhab Ray and  Kallol Paul ]{ Debmalya Sain, Anubhab Ray and  Kallol Paul}

\newcommand{\acr}{\newline\indent}

\address[Sain]{Department of Mathematics\\ Indian Institute of Science\\ Bengaluru 560012\\ Karnataka \\India\\ }
\email{saindebmalya@gmail.com}

\address[Ray]{Department of Mathematics\\ Jadavpur University\\ Kolkata 700032\\ West Bengal\\ INDIA}
\email{anubhab.jumath@gmail.com}

\address[Paul]{Department of Mathematics\\ Jadavpur University\\ Kolkata 700032\\ West Bengal\\ INDIA}
\email{kalloldada@gmail.com}

\thanks{The research of Dr. Debmalya Sain is sponsored by Dr. D. S. Kothari Postdoctoral Fellowship, under the mentorship of Professor Gadadhar Misra. Dr. Debmalya Sain feels elated to acknowledge the extraordinary presence of Mr. Anirban Dey, his beloved brother-in-arms, in his everyday life! The second author would like to thank DST, Govt. of India, for the financial support in the form of doctoral fellowship. The research of third author  is supported by project MATRICS  of DST, Govt. of India.
} 


\begin{abstract}
We explore extreme contractions on finite-dimensional polygonal Banach spaces, from the point of view of   attainment of norm of a linear operator.  We prove that if $ X $ is an $ n- $dimensional polygonal Banach space and $ Y $ is any normed linear space and $ T \in L(X,Y) $ is an extreme contraction, then $ T $ attains norm at $ n $ linearly independent extreme points of $ B_{X}. $ Moreover, if $ T $ attains norm at $ n $ linearly independent extreme points $ x_1, x_2, \ldots, x_n $ of $ B_X $ and does not attain norm at any other extreme point of $ B_X, $ then each $ Tx_i $ is an extreme point of $ B_Y.$ We completely characterize extreme contractions between a finite-dimensional polygonal Banach space and a strictly convex normed linear space. We introduce L-P property for a pair of Banach spaces and show that it has natural connections with our present study. We also prove that for any strictly convex Banach space $ X $ and any finite-dimensional polygonal Banach space $ Y, $ the pair $ (X,Y) $ does not have L-P property. Finally, we obtain a characterization of Hilbert spaces among strictly convex Banach spaces in terms of L-P property. 

\end{abstract}

\maketitle

\section{Introduction.}
The purpose of the present paper is to study the structure and properties of extreme contractions on finite-dimensional real polygonal Banach spaces. Characterization of extreme contractions between Banach spaces is a rich and intriguing area of research. It is worth mentioning that extreme contractions on a Hilbert space are well-understood \cite{Ga, K, N, S}. However, we are far from completely describing extreme contractions between general Banach spaces, although several mathematicians have studied the problem for particular Banach spaces \cite{G, I, Ki, Sh, Sha}. Very recently, a complete characterization of extreme contractions between two-dimensional strictly convex and smooth real Banach spaces has been obtained in \cite{SPM}. To the best of our knowledge, the characterization problem remains unsolved for higher dimensional Banach spaces. Lindenstrauss and Perles carried out a detailed investigation of extreme contractions on a finite-dimensional Banach space in \cite{LP}. It follows from their seminal work that extreme contractions on a finite-dimensional Banach space may have additional properties if the unit sphere of the space is a polytope. Motivated by this observation, we further explore extreme contractions between finite-dimensional polygonal Banach spaces. Without further ado, let us establish the relevant notations and terminologies to be used throughout the paper. \\

In this paper, letters  $X,~Y$  denote real Banach spaces. Let $ B_{X}=\{ x\in X ~:~\|x\| \leq 1 \} $ and
$ S_{ X }=\{ x\in X  ~:~\|x\|=1 \} $ denote the unit ball and the unit sphere of $X$ respectively and let $ L(X,Y) $ be the Banach space of all bounded linear operators from $ X $ to $ Y, $ endowed with the usual operator norm. Let $E_X$ be the collection of all extreme points of the unit ball $B_X.$ We say that a finite-dimensional Banach space $ X $ is a polygonal Banach space if $ S_{X} $ is a polytope, or equivalently, if $ B_{X} $ contains only finitely many extreme points. An operator $ T \in L(X,Y) $ is said to be an extreme contraction if $ T $ is an extreme point of the unit ball $ B_{L(X,Y)} $. We would like to note that extreme contractions are norm one elements of the Banach space $ L(X,Y) $ but the converse is not necessarily true. For a bounded linear operator $ T \in L(X,Y), $ let $ M_T $ be the collection of all unit vectors in $ X $ at which $ T $ attains norm, i.e., 
\[ M_T = \{  x \in S_{X} : \| Tx \| = \| T \|  \}. \]

As we will see in due course of time, the notion of the norm attainment set $ M_T, $ corresponding to a linear operator $ T $ between the Banach spaces $ X $ and $ Y, $ plays a very important role in determining whether $ T $ is an extreme contraction or not. As a matter of fact, we prove that if $ X $ is an $ n- $dimensional polygonal Banach space and $ Y $ is any normed linear space, then $ T \in L(X,Y) $ is an extreme contraction implies that  $ span(M_T \cap E_X) = X. $ Moreover, if $ M_T \cap E_X $ contains  exactly $2n $ elements then $ T( M_T \cap E_X ) \subset E_Y.$ Indeed, this novel connection between extreme contractions and the corresponding norm attainment set is a major highlight of the present paper. As an application of this result, we completely characterize extreme contractions from a finite-dimensional polygonal Banach space to any strictly convex normed linear space. We would like to note that extreme contractions in $ L(l_{1},Y), $ where $ Y $ is any Banach space, have been completely characterized by Sharir in \cite{Sha}. In this paper, we show that extreme contractions in $ L(l_{1}^{n},Y) $ can be characterized using our approach. We also illustrate that the nature of extreme contractions can change dramatically if we consider the domain space to be something other than finite-dimensional polygonal Banach spaces.\\

In \cite{LP}, Lindenstrauss and Perles studied the set of extreme contractions in $ L(X,X), $ for a finite-dimensional Banach space $ X $. One of the main results presented in \cite{LP} is the following:
\begin{theorem}
If $ X $ is a finite-dimensional Banach space, then the following statements are equivalent:\\
$ (1) $ $ T \in E_{L(X,X)},~~ x \in  E_X \Rightarrow Tx \in E_X. $\\
$ (2) $ $ T_1,T_2 \in E_{L(X,X)} \Rightarrow T_1 \circ T_2 \in E_{L(X,X)}. $\\
$ (3) $ $ \{T_i\}_{i=1}^{m} \subseteq E_{L(X,X)} \Rightarrow \|T_1 \circ \ldots \circ T_m \|=1, $ for all $ m. $
\end{theorem} 

Furthermore, they also proved in \cite{LP} that if $ X $ is a Banach space of dimension strictly less than $ 5 $ then $ X $ has the Properties $ (1), $ $ (2) $ and $ (3), $ as mentioned in the above theorem, if and only if either of the following is true:\\
\noindent $ (i) $ $ X $ is an inner product space.\\
\noindent $ (ii) $ $ B_{X} $ is a polytope with the property that for every facet (i.e., maximal proper face) $ K $ of $ B_{X}, $ $ B_{X} $ is the convex hull of $ K \bigcup (-K). $\\

In view of the results obtained in \cite{LP}, it seems natural to introduce the following definition in the study of extreme contractions between Banach spaces: 

\begin{definition}
Let $ X,~Y $ be Banach spaces. We say that the pair $ (X,Y) $ has L-P (abbreviated form of Lindenstrauss-Perles) property if a norm one bounded linear operator $ T \in L(X,Y) $ is an extreme contraction if and only if $ T(E_X) \subseteq E_Y. $
\end{definition}

We would like to remark that although the study conducted in \cite{LP} was only for the case $ X=Y, $ we have consciously formulated the above mentioned definition in a broader context, by not imposing the restriction that the domain space and the range space must be identical. In order to illustrate that our definition is meaningful, we provide examples of Banach spaces $ X,~Y $ such that $ X \neq Y $ and the pair $ (X,Y) $ has L-P property. As a matter of fact, we observe that the pair $ (l_{\infty}^{3},l_{1}^{3}) $ has L-P property. It is apparent that our study in this paper is motivated by the investigations carried out in \cite{LP}. We further give examples of finite-dimensional polygonal Banach spaces $ X,~Y $ such that the pair $ (X,Y) $ does not have L-P property.  We prove that if $ X $ is a strictly convex Banach space and $ Y $ is a finite-dimensional polygonal Banach space, then the pair $ (X,Y) $ does not have L-P property. We end the present paper with a characterization of Hilbert spaces among strictly convex Banach spaces in terms of L-P property.

\section{ Main Results.}
Let us first make note of the following easy proposition:

\begin{prop}\label{proposition:bounded}
Let $ X $ be an $ n- $dimensional Banach space and let $ A \subseteq X $ be a bounded set. Suppose 
$ \{ x_1, x_2, \ldots , x_n\}$ is a basis of $ X. $ Then the following set $ S= \{ |\alpha_i| : z=\sum\limits_{i=1}^{n} \alpha_i x_i~~ \mbox{and}~~ z \in A \} $ is bounded.
\end{prop}

Using the above proposition, we obtain a necessary condition for a linear operator on a polygonal Banach space to be an extreme contraction, in terms of the norm attainment set. 

\begin{theorem}\label{theorem:dimension-n}
Let $ X $ be an $ n- $dimensional polygonal Banach space and let $ Y $ be any normed linear space. Let $ T \in L(X,Y) $ be an extreme contraction. Then $ span(M_T \cap E_X) = X. $ Moreover, if $ M_T \cap E_X $ contains  exactly $2n $ elements then $ T( M_T \cap E_X ) \subset E_Y.$
\end{theorem}

\begin{proof}
First, we prove that  $ span(M_T \cap E_X) = X. $ Since $ X $ is an $ n- $dimensional Banach space, an easy application of Krein-Milman theorem ensures that $ B_{X} $ has at least $ 2n $ extreme points out of which $ n $ are linearly independent. If $ M_T $ contains $ n $ linearly independent extreme points then we are done. Suppose that $ M_T $ does not contain $ n $ linearly independent extreme points. Let $ M_T $ contains at most $ k $ linearly independent extreme points, where $ k < n. $ Let $ \{x_1, x_2, \ldots, x_k\} $ be one such linearly independent subset of $ M_T, $ consisting of extreme points of $ B_X. $ We extend it to a basis  $ \{ x_1, x_2, \ldots,x_k,x_{k+1}, \ldots , x_n\} $ of $ X $ such that each $ x_i, ~~ i=1, 2, \ldots, n $ is an extreme point of $ B_X. $ 
Now, we can choose $ w(\neq 0) \in B_Y $ such that $ T{x_n} \pm \frac{w}{j} \neq 0 $ for every $ j \in \mathbb{N}. $ 
For each $ j \in \mathbb{N}, $ define linear operators $ T_j, ~~ S_j : X \rightarrow Y $ as follows:

\begin{equation*}
\begin{aligned}[c]
T_j(x_i) &=T{x_i} ~ (i =1,2,\ldots, n-1)\\
T_j(x_n) &=Tx_n + \frac{w}{j}
\end{aligned}
\qquad \qquad
\begin{aligned}[c]
S_j(x_i) & =  T{x_i} ~(i = 1,2,\ldots, n-1)\\
S_j(x_n) & =  Tx_n - \frac{w}{j}.
\end{aligned}
\end{equation*}
Then $ T \neq T_j $ and $ T \neq S_j $ for all $ j \in \mathbb{N}. $ Also, we have $ T= \frac{1}{2}(T_j+S_j) $ for all
$ j \in \mathbb{N}. $\\
Since $ T $ is an extreme contraction, there exists a subsequence $\{j_k\}$ such that for all $ j_k, $ either $ \|T_{j_k}\| > 1 $ or $ \|S_{j_k}\| > 1. $ Without loss of generality, we may and do assume that $ \|T_j\| > 1 $ for all $ j \in \mathbb{N}.$\\
Since $ X $ is finite-dimensional, each $ T_j $ attains norm at an extreme point of $ B_X. $ Suppose $ \|T_j{y_j}\|= \|T_j\| > 1 $ where each $ y_j $ is an extreme point of $ B_X. $\\
Let $ z = a_1x_1+\ldots+a_{n-1}x_{n-1}+a_nx_n \in S_X $ be arbitrary, then by Proposition \ref{proposition:bounded}, there exists $ r>0 $ such that $ |a_n| \leq r, $ for each $ z \in S_X.$\\
Now, $ \| (T_j-T)z\|=\|(T_j-T)(a_1x_1+\ldots+a_nx_n)\|=\|\frac{a_nw}{j}\| \leq \frac{r}{j} \rightarrow 0 $ as $ j \rightarrow \infty.$\\
So, $ T_j \rightarrow T $ as $ j \rightarrow \infty. $ As each $ y_j \in S_X $ and $ S_X $ is compact, we may assume without loss of generality that 
$ y_j \rightarrow y_0 $ as $ j \rightarrow \infty, $ where $ y_0 \in S_X.$\\
Noting that there are only finitely many extreme points and each  $ y_j $ is an extreme point, we can further assume that $ y_j=y_0 $ for each $ j. $ So, $ T_j{y_j}=T_j{y_0}\rightarrow Ty_0 $ as $ j \rightarrow \infty. $\\
Therefore, $ \|Ty_0\| \geq 1. $ As $ \|T\|=1, $ $ \|Ty_0\|=1. $ So, $ y_0 \in M_T .$ Since $ y_0 $ is an extreme point and $ M_T $ has at  most $ k $ linearly independent extreme points, the set $ \{ y_0, x_1, x_2, \ldots, x_k \} $ is linearly dependent. Let $ y_0=a_1x_1+a_2x_2+\ldots + a_k x_k. $ Then, 
$ T_j{y_j}=T_j{y_0}=T_j(a_1x_1+\ldots+a_kx_k)=a_1Tx_1+\ldots+a_kTx_k=Ty_0. $
Therefore, $ \|T_j{y_j}\|=\|Ty_0\|=1, $ which is contradiction to the fact that $\|T_j\| > 1 $ for all $j \in \mathbb{N}. $ Thus, if $ T $ is an extreme contraction, then $ span(M_T \cap E_X) = X. $\\
Now, we prove the next part of the theorem. As $ X $ is a finite-dimensional polygonal Banach space, $ E_X $ is a finite set. Let $ \{ \pm{x_1}, \pm{x_2}, \ldots, \pm{x_m} : ~~ m\geq n \} =E_X $ and $ (M_T \cap E_X) = \{ \pm{x_i} : ~~i=1,2,\ldots,n \}. $\\
As $ T $ attains norm only at $ \{ \pm{x_i} : ~~i=1,2,\ldots,n \}, $ $ \|T\|=\|Tx_i\|=1 $ for all $ i=1,2,\ldots,n $ and $ \|Tx_j\| < \|T\|=1 $ for all $ j=n+1,n+2,\ldots,m. $ Therefore, there exists $ \epsilon > 0 $ such that $ \|Tx_j\| < 1-\epsilon $ for all $ j=n+1, n+2, \ldots ,m. $\\
Clearly, $ \{x_1, x_2, \ldots,x_n\} $ is a basis of $ X. $ Therefore, there exist scalars $ a_{1j}, a_{2j}, \ldots , a_{nj} $ such that
\[ x_j=a_{1j}x_1+a_{2j}x_2+\ldots+a_{nj}x_n, ~\forall ~ j=n+1,n+2,\ldots,m. \]
Now, using Proposition \ref{proposition:bounded}, there exists $ r>0 $ such that $ |a_{ij}| \leq r $ for all $ i=1,2,\ldots,n $ and $ j=n+1,n+2,\ldots,m. $\\
Suppose $ Tx_k $ is not an extreme point of $ B_Y $ for some $ k \in \{ 1,2,\ldots,n\}. $ Then there exist $ y,z \in S_Y \cap B[Tx_k, \frac{\epsilon}{2r}]$ such that $ Tx_k= (1-t)y+tz $ for some $ t \in (0,1). $\\
Let us define $ T_1:X \rightarrow Y $ by $ T_1{x_i}=Tx_i $ for all $ i=1,2, \ldots,n ~\mbox{and}~ i\neq k $ and $ T_1{x_k}=y, ~i=k. $\\
Similarly, we define $ T_2:X \rightarrow Y $ by $ T_2{x_i}=Tx_i $ for all $ i=1,2, \ldots,n ~\mbox{and}~ i\neq k $ and $ T_2{x_k}=z, ~i=k. $\\
Clearly, $ T_1 \neq T $ and $ T_2 \neq T. $ Also, we have $ T=(1-t)T_1+tT_2. $ We claim that $ \|T_1\|=\|T_2\|=1. $\\
Clearly, $ \|T_1{x_i}\|=\|Tx_i\|=1 $ for all $ i=1,2,\ldots,n ~\mbox{and}~ i\neq k. $ Also, $ \|T_1{x_k}\|=\|y\|=1 $ as $ y \in S_Y \cap B[Tx_k, \frac{\epsilon}{2r}].$\\
Furthermore, for each $ j=n+1,\ldots,m, $ we have,
\begin{eqnarray*}
        \|T_1{x_j}\| & = & \|a_{1j}T_1{x_1}+\ldots+a_{kj}T_1{x_{k}}+\ldots +a_{nj}T_1{x_n}\| \\
				             & = & \|a_{1j}T{x_1}+\ldots+a_{kj}y+\ldots +a_{nj}T{x_n}\| \\
							       & = & \|a_{1j}T{x_1}+\ldots +a_{nj}T{x_n}+a_{kj}(y-T{x_k})\| \\
									& \leq & \|T{x_j}\| + |a_{kj}|\|y-T{x_k}\| \\
									& \leq & 1- \epsilon + r\frac{\epsilon}{2r} \\
									   & = & 1-\frac{\epsilon}{2} < 1.
\end{eqnarray*}
Thus, for any extreme point $ x_i $  of $ B_X, $ $ \|T_1{x_i}\| \leq 1 $ and so $ \|T_1\|=1. $
Similarly, $ \|T_2\|=1.$\\
Therefore, $ T=(1-t)T_1+tT_2 $ and $\|T_1\|=\|T_2\|=1. $ This contradicts the fact that $ T $ is an extreme contraction. Thus, if $ M_T \cap E_X $ contains  exactly $2n $ elements then $ T( M_T \cap E_X ) \subset E_Y.$ This completes the proof of the theorem. 
\end{proof}

As an immediate application of Theorem \ref{theorem:dimension-n}, it is possible to characterize extreme contractions from a finite-dimensional polygonal Banach space to a strictly convex normed linear space. We present the characterization in the form of the next theorem.   

\begin{theorem}\label{theorem:poly}
Let $ X $ be a finite-dimensional polygonal Banach space and let $ Y $ be a strictly convex normed linear space. Then $ T \in L(X,Y), $ with $ \|T\|=1, $ is an extreme contraction if and only if $ span (M_T \cap E_X) = X.$
\end{theorem}
\begin{proof}
The proof of the necessary part of the theorem follows directly from Theorem  \ref{theorem:dimension-n}. Here, we prove the sufficient part of the theorem. Suppose that $ T $ is not an extreme contraction. Then there exist $ T_1, T_2 \in L(X,Y) $ such that $ T_1, T_2 \neq T, $ $ \|T_1\|=\|T_2\|=1 $ and $ T = (1-t)T_1+tT_2 $ for some $ t \in (0,1). $ Let $ \{ x_1, x_2, \ldots, x_n \} \subseteq M_T \cap E_X $ be a basis of $ X. $ Then $ Tx_i = (1-t)T_{1}x_{i}+ t T_{2}x_{i}, $ for each $ i \in \{1,2, \ldots, n \}. $ We also note that $ T_{1}x_{i}, T_{2}x_{i} \in B_Y, $ as $ \|T_1\|=\|T_2\|=1. $ As $ Y $ is strictly convex, we have $ E_Y=S_Y. $ Therefore, each $ Tx_{i} $ is an extreme point of $ B_{Y} $ and it follows from that $ Tx_{i}= T_{1}x_{i} = T_{2}x_{i}, $ for each $ i \in \{1,2, \ldots, n \}. $ However, this implies that $ T_1, T_2 $ agree with $ T $ on a basis of $ X $ and therefore, $ T_1 = T_2 = T.$ This contradicts our initial assumption that $ T_1, T_2 \neq T. $ Thus $ T $ is an extreme contraction and this completes the proof of the theorem.
\end{proof}

As another useful application of Theorem \ref{theorem:dimension-n}, it is possible to  characterize extreme contractions from $ l_1^n $ to $ Y, $ where $ Y $ is any normed linear space. It is worth mentioning that the following corollary also follows from \cite[Lemma 2.8]{Sha}.
\begin{cor}\label{corollary:l_1}
Let $ X=l_{1}^{n} $ and let $ Y $ be any normed linear space. Let $ T \in L(X,Y) $ with $\|T\|=1.$ Then $ T $ is extreme contraction if and only if $ M_T=E_X $ and $ T(E_X) \subseteq E_Y. $  
\end{cor}
\begin{proof}
Firstly, we know that $ \{ \pm e_1, \pm e_2, \ldots, \pm e_n \} $ is the set of all extreme points of the unit ball of $ l_{1}^{n},$ where $e_i=(\underbrace{0,0,\ldots,0,1,}_{i}0,\ldots,0)$ for each $ i \in \{1,2, \ldots, n\}. $ Now, the proof of the necessary part of the theorem follows directly from Theorem  \ref{theorem:dimension-n}. Here, we prove the sufficient part of the theorem. Suppose that $ T $ is not an extreme contraction. Then there exist $ T_1, T_2 \in L(X,Y) $ such that $ T_1, T_2 \neq T, $ $ \|T_1\|=\|T_2\|=1 $ and $ T = (1-t)T_1+tT_2 $ for some $ t \in (0,1). $ Then $ Te_i = (1-t)T_{1}e_{i}+ t T_{2}e_{i}, $ for each $ i \in \{1,2, \ldots, n \}. $  We also note that for each $ i \in \{ 1,2, \ldots, n \}, $ $ T_{1}e_{i}, T_{2}e_{i} \in B_Y, $ as $ \|T_1\|=\|T_2\|=1. $ Since $ Te_{i} \in E_Y, $ it follows that $ Te_{i}= T_{1}e_{i} = T_{2}e_{i}, $ for each $ i \in \{1,2, \ldots, n \}. $ However, this implies that $ T_1, T_2 $ agree with $ T $ on a basis of $ X $ and therefore, $ T_1 = T_2 = T.$ This contradicts our initial assumption that $ T_1, T_2 \neq T. $ Thus $ T $ is an extreme contraction and this completes the proof of the corollary. 
\end{proof}
 
\begin{remark}
If $ Y $ is a polygonal Banach space with $ p $ pair of extreme points then the number of extreme contractions in $ L(l_{1}^{n},Y) $ is $ (2p)^n. $ Moreover, since the number of extreme contractions in $ L(l_{1}^{n},Y) $ is finite, each extreme point of the unit ball of $ L(l_{1}^{n},Y) $ is also an exposed point of the unit ball of $ L(l_{1}^{n},Y). $ Therefore, the number of exposed points of the unit ball of $ L(l_{1}^{n},Y) $ is also $ (2p)^n. $ In particular,\\
\noindent (i) the number of extreme contractions in $ L(l_{1}^{n},l_{1}^{n}) $ is $ {2^n}{n^n}. $\\
\noindent (ii) the number of extreme contractions in $ L(l_{1}^{n}, l_{\infty}^{n}) $ is $ 2^{n^2}. $\\
\end{remark}

As illustrated in \cite{LP}, one of the most intriguing aspects of the study of extreme contractions between Banach spaces is to explore the extremity of the images of the extreme points of the unit ball of the domain space under extreme contractions. Till now, we have considered only finite-dimensional polygonal Banach spaces as the domain space. However, if we choose the domain space to be something other than the polygonal Banach space then the scenario changes drastically. In fact, choosing the Euclidean space $ l_2^n $ as the domain space and $ l_{\infty}^n $ as the co-domain space, we have the following proposition:

\begin{prop}\label{proposition:l_2}
Let $ T \in L(X,Y) $ with $ \|T\|=1,$ where $ X=l_{2}^n $ and $ Y=l_{\infty}^n. $ Then $ T $ is an extreme contraction if and only if corresponding matrix representation of $ T $ with respect to standard ordered basis is of the form
\[ 
\begin{bmatrix}
a_{11}&a_{12}&\cdots &a_{1n} \\
a_{21}&a_{22}&\cdots &a_{2n} \\
\vdots & \vdots & \ddots & \vdots\\
a_{n1}&a_{n2}&\cdots &a_{nn}
\end{bmatrix}\]
where $ \sum \limits_{j=1}^{n} {a_{ij}}^2=1 $ for all $ i=1,2,\ldots,n.$
\end{prop}

\begin{proof}
We first note that a bounded linear operator $ T $ between reflexive Banach spaces is an extreme contraction if and only if $ T^* $ is an extreme contraction, which follows from the facts that $ T = (1-t)T_1+ t T_2 \Leftrightarrow T^*=(1-t)T_{1}^{*} + t T_{2}^{*}, $ $ \|T\|=\|T^*\| $ and $ (T^*)^*=T $ by reflexivity of domain and co-domain spaces. Now, the proof of the proposition directly follows from Corollary \ref{corollary:l_1}. 
\end{proof}

In the following two examples, we further illustrate the variation in the norm attainment set of extreme contractions, depending on the domain and the range space.

\begin{example}
In $ L(l_{2}^2, l_{\infty}^2), $ \[ T=
\begin{bmatrix}
1&0 \\
1&0
\end{bmatrix}\] is an extreme contraction though $ T $ attains norm only at $ \pm(1,0).$ 
\end{example}
\begin{example}
In $ L(l_{2}^2, l_{\infty}^2), $ \[ T=
\begin{bmatrix}
1&0 \\
0&1
\end{bmatrix}\] is an extreme contraction. Here, $ T $ attains norm at $ \pm(1,0)$ and $ \pm(0,1) $ but $T(1,0)=(1,0)$ and $T(0,1)=(0,1) $ are not extreme points of unit ball of $l_{\infty}^2.$   
\end{example}

Let us now focus on the connections between the study carried out in \cite{LP} and some of the results obtained by us in the present paper, in light of the newly introduced notion of L-P property for a pair of Banach spaces $ (X,Y). $ It is easy to observe that $ l_{1}^{n} $ is a universal L-P space, in the sense that the pair $ (l_{1}^{n},Y) $ has L-P property for every Banach space $ Y. $ We would also like to note that it follows from the works \cite{Li, L} of Lima that there exist finite-dimensional polygonal Banach spaces $ X,~Y, $ with $ X \neq Y, $ such that the pair $ (X,Y) $ has L-P property. Also, there are finite-dimensional polygonal Banach spaces $ X, Y $ with $ X \neq Y, $ such that the pair $ (X,Y) $ does not satisfy L-P property. As for example, the pair  $ (l_{\infty}^{3},l_{1}^{3}) $ has L-P property, which follows from  \cite[Theorem 2.1]{L}, whereas the pair $ (l_{\infty}^{4},l_{1}^{4}) $ does not satisfy L-P property, which follows from Lemma $ 3.2 $ of \cite{L}. \smallskip 

Let us observe that Proposition \ref{proposition:l_2} implies, in particular, that the pair $ (l_{2}^{n}, l_{\infty}^{n}) $ does not have L-P property. Now, we prove that this observation is actually a consequence of the following general result.

\begin{theorem}\label{theorem:L-Pstrict}
Let $ X $ be any strictly convex Banach space and $ Y $ be a finite-dimensional polygonal Banach space. Then the pair $ (X,Y) $ does not have L-P property.
\end{theorem}
\begin{proof}
Suppose that the pair $ (X,Y) $ satisfies L-P property. Then for any extreme contraction $ T \in L(X,Y), $ we have $ T(E_X) \subseteq E_Y. $ As $ X $ is strictly convex, $ B_X $ has infinitely many extreme points, in fact, $ E_X=S_X. $ On the other hand, as $ Y $ is a finite-dimensional polygonal Banach space, $ B_Y $ has only finitely many extreme points. Thus, for any extreme contraction $ T \in L(X,Y), $ $ T $ cannot be an one-to-one operator. Therefore, there exists a nonzero $ x \in B_X $ such that $ T(x)=0. $ So, $ T(\frac{x}{\|x\|})=0, $ which contradicts the fact that $ T(E_X)= T(S_X)\subseteq E_Y. $ This completes the proof of the fact that the pair $ (X,Y) $ does not have L-P property.
\end{proof}

\begin{remark}
From Theorem \ref{theorem:L-Pstrict}, we can conclude that the pairs $ (l_{p}^{n},l_{1}^{n}) $ and $ (l_{p}^{n},l_{\infty}^{n}) $ do not have L-P property for any $ 1 < p < \infty. $ Furthermore, we already know that the pair $ (l_{1}^{n},Y) $ has L-P property for every Banach space $ Y. $ This illustrates that there exist Banach spaces $ X, ~~Y $ such that the pair $ (X,Y) $ has L-P property but the pair $ (Y,X) $ does not have L-P property. 
\end{remark}
\begin{remark}
 The above Theorem \ref{theorem:L-Pstrict} holds for any Banach space $Y$ with countably many extreme points.
\end{remark}

Now, we give a characterization of finite-dimensional Hilbert spaces among all finite-dimensional strictly convex Banach spaces in terms of L-P property.

\begin{theorem}\label{theorem:Hilbert-finite}
Let $ X $ be a finite-dimensional strictly convex Banach space. Then $ X $ is a Hilbert space if and only if the pair $ (X, X) $ has L-P property.
\end{theorem}
\begin{proof}
Firstly, we claim that for a strictly convex Banach space $ X, $ the pair $ (X, X) $ has L-P property if and only if  every extreme contraction in $ L(X,X) $ is an isometry. Clearly, if every extreme contraction in $ L(X,X) $ is an isometry then the pair $ (X,X) $ has L-P property, as for a strictly convex space $ X,$ $ S_X=E_X. $ Conversely, since $ (X,X) $ satisfies L-P property, $ T \in L(X,X) $ is an extreme contraction if and only if $ T(E_X) \subseteq E_X. $ As $ X $ is strictly convex, $ E_X = S_X. $  Thus, $ T \in L(X,X) $ is extreme contraction if and only if $ T(S_X) \subseteq S_X. $ In other words, we must have $ M_T=S_X $ and $ \|T\|=1, $ i.e., $ T \in L(X,X) $ is an isometry. This completes the proof of our claim. On the other hand, in \cite{N}, Navarro proved that a finite-dimensional Banach space $ X $ is a Hilbert space if and only if every extreme contraction in $ L(X,X) $ is an isometry. The proof of the theorem follows directly from a combination of these two facts. 
\end{proof}

As an application of Theorem \ref{theorem:Hilbert-finite}, we characterize  Hilbert spaces among all  strictly convex Banach spaces in terms of L-P property.

\begin{cor}
A strictly convex Banach space $ X $ is a Hilbert space if and only if for every two-dimensional subspace $ Y $ of $ X, $ the pair $ (Y, Y) $ has L-P property.
\end{cor}
\begin{proof}
Let us first prove the ``only if" part. If $ X $ is a Hilbert space then every two-dimensional subspace $ Y $ of $ X $ is also a Hilbert space. Again, we know that for a finite-dimensional Hilbert space $ Y, $ isometries are the only extreme contractions in $ L(Y,Y) $ and $ E_Y = S_Y. $ Thus, for every two-dimensional subspace $ Y $ of $ X, $ the pair $ (Y,Y) $ has L-P property.\\ 
Let us now prove the ``if" part. If $ X $ is a strictly convex Banach space then every two-dimensional subspace $ Y $ of $ X $ is also a strictly convex Banach space. Since the pair $ (Y,Y) $ has L-P property, then by Theorem \ref{theorem:Hilbert-finite}, $ Y $ is a Hilbert space. Thus every two-dimensional subspace of $ X $ is a Hilbert space which means that $ X $ itself has to be a Hilbert space. This establishes the theorem.\\
\end{proof}

\begin{acknowledgement}
We would like to thank an anonymous mathematician for the statements of Theorem \ref{theorem:dimension-n} and Theorem \ref{theorem:poly} in their present form.
\end{acknowledgement}

\bibliographystyle{amsplain}

\begin{thebibliography}{99}

\bibitem{G} R. Grz\c a\' slewicz,  \textit{Extreme operators on 2-dimensional lp-spaces}, Colloq. Math., \textbf{44} (1981), 309--315. 

\bibitem{Ga} R. Grz\c a\' slewicz,  \textit{Extreme contractions on real Hilbert spaces},  Math. Ann., \textbf{261} (1982), 463-466.

\bibitem{I} A. Iwanik, \textit{Extreme contractions on certain function spaces}, Colloq. Math., \textbf{40} (1978), 147--153.
 
\bibitem{K} R. V. Kadison,  \textit{Isometries of operator algebras}, Ann. of Math. (2), \textbf{54} (1951), 325-338.

\bibitem{Ki} C. W. Kim, \textit{Extreme contraction operators on $ l_{\infty} $}, Math. Z., \textbf{151} (1976), 101--110.
	
\bibitem{Li}	\AA. Lima, \textit{Intersection properties of balls in spaces of compact operators}, Ann. Inst. Fourier, \textbf{28} (1978), 35--65.
	
\bibitem{L}	\AA. Lima, \textit{On extreme operators on finite-dimensional Banach spaces whose unit balls are polytopes}, Ark. Mat., \textbf{19} (1981), no. 1, 97--116.


\bibitem{LP} J. Lindenstrauss and M. A. Perles, \textit{On extreme operators in finite-dimensional spaces}, Duke Math. J.,
  \textbf{36} (1969), 301--314.
	
\bibitem{N} M. A.  Navarro, \textit{Some characterizations of finite-dimensional Hilbert spaces}, J. Math. Anal. Appl.,
  \textbf{223} (1998), 364--365.	

\bibitem{S} D. Sain, \textit{On extreme contractions and the norm attainment set of a bounded linear operator}, Ann. Funct. Anal., Accepted, arXiv:1708.07333v1 [math.FA], 24 Aug 2017.

\bibitem{SPM} D. Sain, K. Paul and A. Mal, \textit{On extreme contractions between real Banach spaces}, arXiv:1801.09980v1 [math.FA], 30 Jan 2018.
	
\bibitem{Sh} M. Sharir, \textit{Characterization and properties of extreme operators into C(Y)}, Israel J. Math., \textbf{12} (1972), 174--183.
	
\bibitem{Sha} M. Sharir, \textit{Extremal structure in operator spaces}, Trans. Amer. Math. Soc., \textbf{186} (1973), 91--111.	



\end{thebibliography}

\end{document}